\documentclass[12pt]{article}
\usepackage{amssymb,amsmath,graphics}
\usepackage{amsfonts}                   
\usepackage{amsthm}                     
\usepackage{times}         
\usepackage{epsfig}             

\title{Generalizing the Converse to Pascal's Theorem via Hyperplane
  Arrangements and the Cayley-Bacharach Theorem}

\author{
Will Traves \thanks{Department of Mathematics, U.S. Naval Academy,
traves@usna.edu}}

\theoremstyle{plain}                         
\newtheorem{theorem}{Theorem}

\newtheorem{lemma}[theorem]{Lemma}

\theoremstyle{definition}

\newtheorem{question}[theorem]{Question}

\newtheorem{exercise}[theorem]{Exercise}

\newcommand{\C}{\mathbb{C}}

\newcommand{\V}{\mathbb{V}}
\newcommand{\R}{\ensuremath \mathbb{R}}
\renewcommand{\P}{\ensuremath \mathbb{P}}

\newcommand{\Sec}{\ensuremath{ {\text Sec}}}
\newcommand{\dedication}[1]{\begin{center} {\em #1} \end{center}} 

\begin{document}
\maketitle 

\abstract{Using a new point of view inspired by hyperplane
  arrangements, we generalize the converse to Pascal's Theorem,
  sometimes called the Braikenridge-Maclaurin Theorem. In particular,
  we show that if $2k$ lines meet a given line, colored green, in $k$
  triple points and if we color the remaining lines so that each
  triple point lies on a red and blue line then the points of
  intersection of the red and blue lines lying off the green line lie
  on a unique curve of degree $k-1$.  We also use these ideas to
  extend a second generalization of the Braikenridge-Maclaurin
  Theorem, due to M\"obius.  Finally we use Terracini's Lemma and
  secant varieties to show that this process constructs a dense set of
  curves in the space of plane curves of degree $d$, for degrees $d
  \leq 5$. The process cannot produce a dense set of curves in higher
  degrees. The exposition is embellished with several exercises
  designed to amuse the reader. }

\dedication{Dedicated to H.S.M. Coxeter, who demonstrated a heavenly
  syzygy:  \\ the
  sun and moon aligned with the Earth, through a pinhole. \\ (Toronto,
  May 10, 1994, 12:24:14)} \smallskip

\begin{section}{Introduction}

In Astronomy the word {\em syzygy} refers to three celestial bodies that lie
on a common line. More generally, it sometimes is used to describe
interesting geometric patterns. For example, in a triangle the three median lines
that join vertices to the midpoints of opposite sides meet in a
common point, the centroid, as illustrated in Figure
\ref{trianglesyz}. Choosing coordinates, this fact can be
viewed as saying that three objects lie on a line: there is a linear
dependence among the equations defining the three median lines.  In
Commutative Algebra and Algebraic Geometry, a syzygy refers to any
relation among the generators of a module. 

\begin{figure}[h!t]
\begin{center}
\scalebox{0.8}{\includegraphics[width=0.75\textwidth]{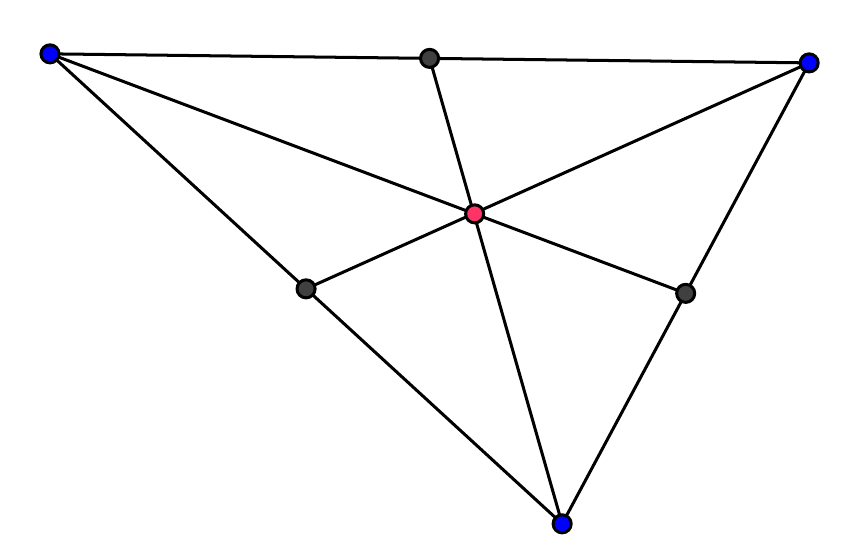}}
\end{center}
\caption{The three medians of a triangle intersect at the centroid. }
\label{trianglesyz}
\end{figure}

Pappus's Theorem, which dates from the fourth century A.D., describes
another syzygy. It is one of the inspirations of modern projective
geometry.

\begin{theorem}[Pappus] \label{theorem:Pappus} If 3 points $A,
  B, C$ lie on one line, and three points $a, b, c$ lie on another,
  then the lines $Aa$, $Bb$, $Cc$ meet the lines $aB$, $bC$, $cA$ in
  three new points and these new points are collinear, as illustrated
  in the left diagram of Figure \ref{PappusPascal}. 
\end{theorem}

Pappus's Theorem
  appears in his text, {\it Synagogue} \cite{Jones1,Jones2}, a collection of classical Greek geometry with insightful commentary. David Hilbert observed that Pappus's
  Theorem is equivalent to the claim that the multiplication of
  lengths is commutative (see e.g. Coxeter
  \cite[p. 152]{Coxeter}). Thomas Heath
  believed that Pappus's intention was to revive the
  geometry of the Hellenic period \cite[p. 355]{Heath}, but it wasn't until 1639 that the
  sixteen year-old Blaise Pascal generalized
  Pappus's theorem  \cite[Section 3.8]{CoxeterGreitzer}, replacing the
  two lines with a more general conic section. 

\begin{theorem}[Pascal] \label{theorem:Pascal} If 6 points $A,
  B, C, a, b, c$ lie on a conic section, 
  then the lines $Aa$, $Bb$, $Cc$ meet the lines $aB$, $bC$, $cA$ in
  three new points and these new points are collinear, as illustrated
  in the right diagram of Figure \ref{PappusPascal}.  
\end{theorem} 

\begin{figure}[h!t]
 \hspace{0.05\textwidth}
\includegraphics[width=0.40\textwidth]{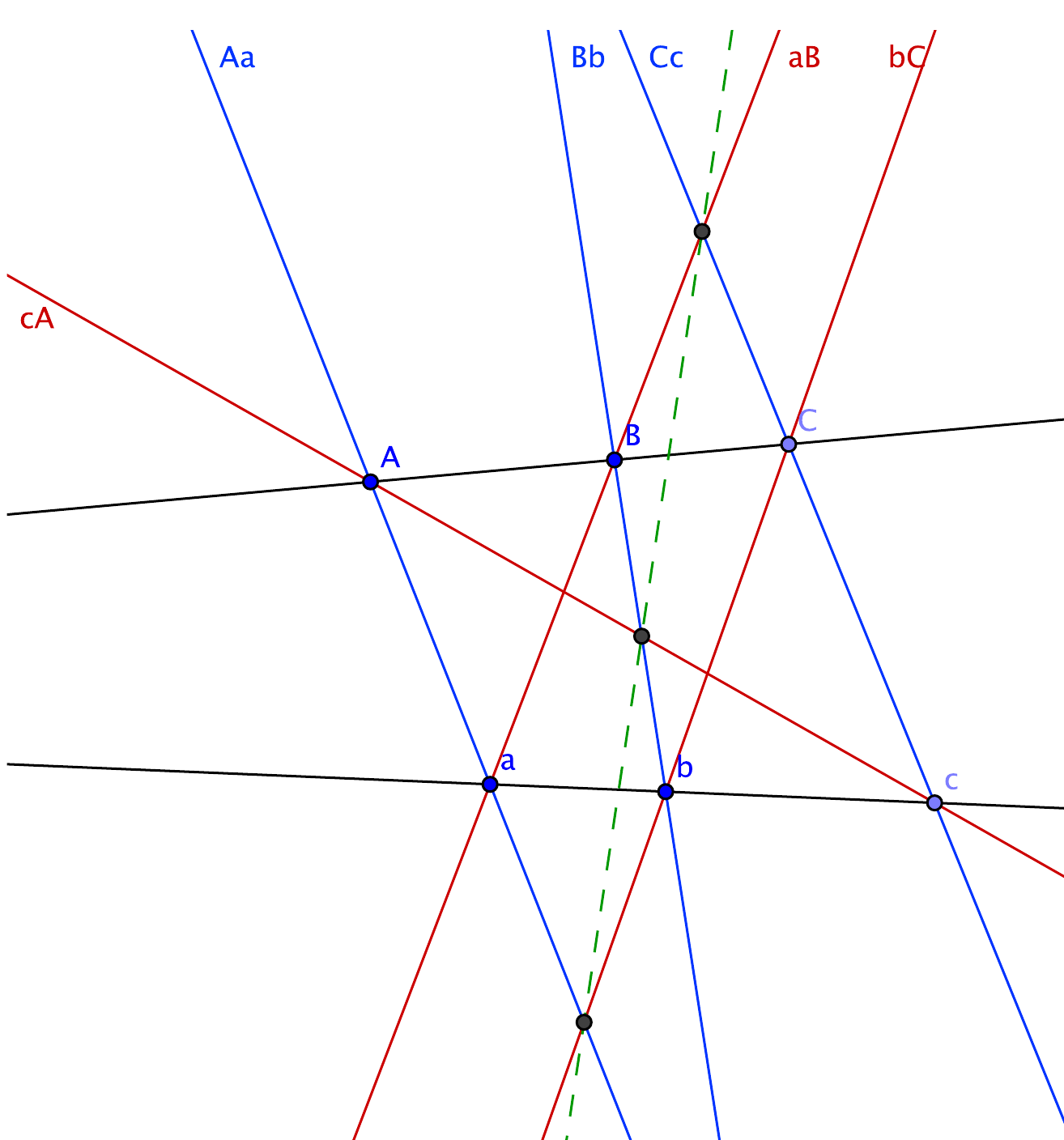} \hspace{0.05\textwidth}
\scalebox{1.05}{\includegraphics[width=0.40\textwidth]{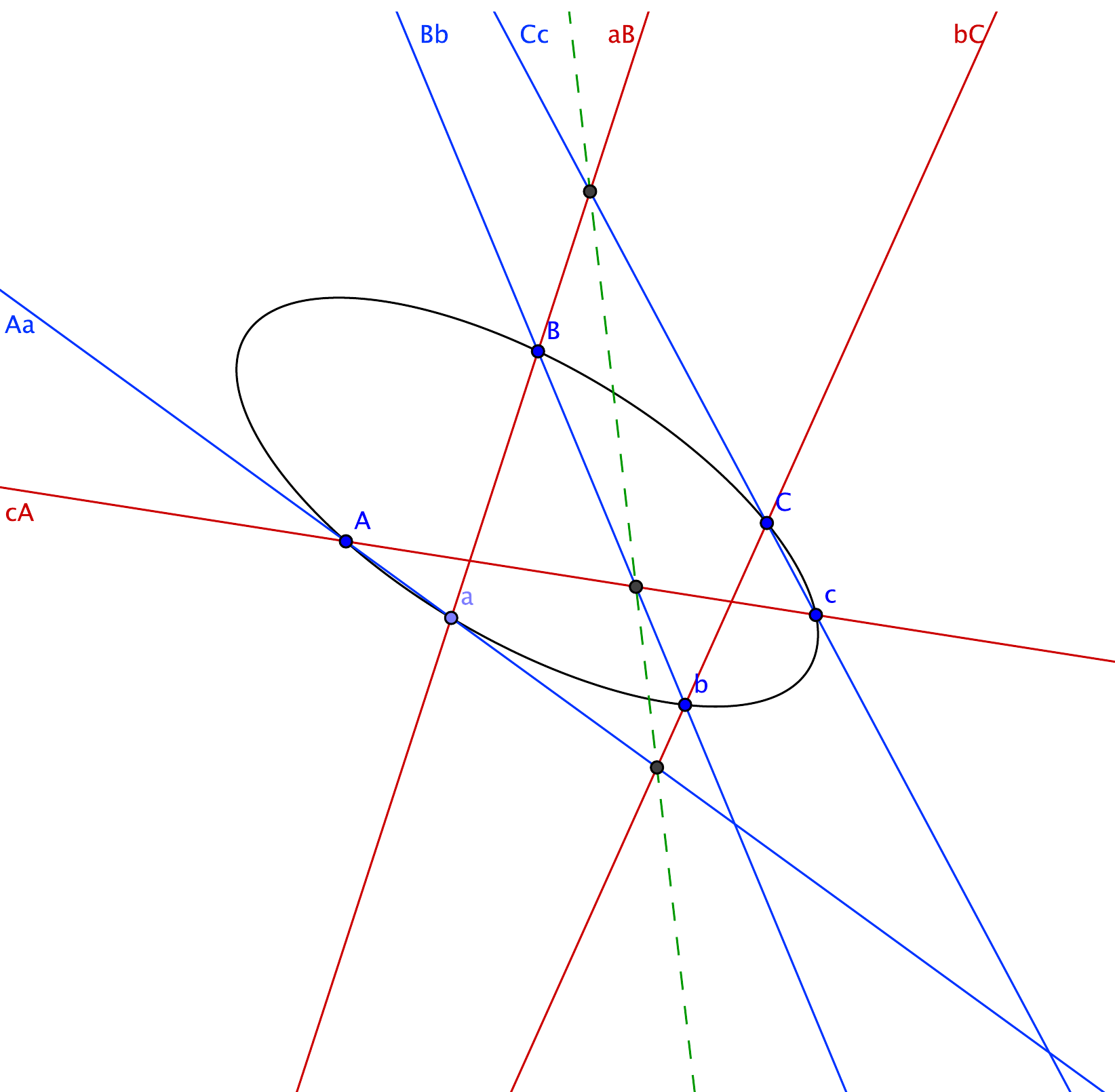}}
\caption{Illustrations of Pappus's Theorem (left) and Pascal's Theorem
(right). }
\label{PappusPascal}
\end{figure}

Pascal's theorem is sometimes formulated as the Mystic
Hexagon Theorem: if a hexagon is inscribed in a conic then the 3
points lying on lines extending from pairs of opposite edges of the
hexagon are collinear, as in Figure \ref{MHT}.

\begin{figure}[h!t]
\begin{center}
\scalebox{0.8}{\includegraphics[width=0.75\textwidth]{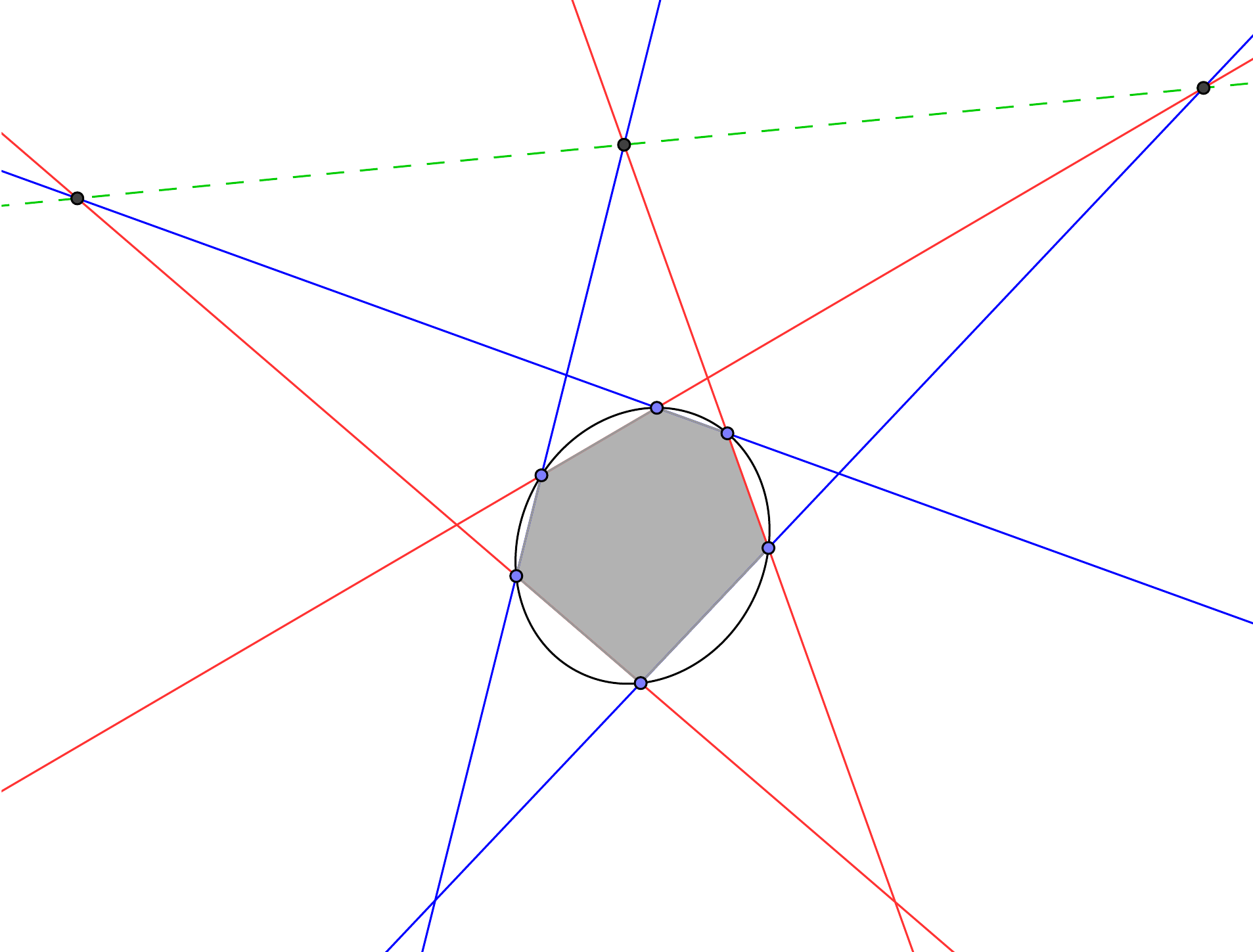}}
\end{center}
\caption{The Mystic Hexagon Theorem. }
\label{MHT}
\end{figure}

It is not clear why the theorem deserves the
adjective mystic. Perhaps it refers to the case where a regular
hexagon is inscribed in a circle. In that case, the three pairs of opposite
edges are parallel and the theorem then predicts that the parallel
lines should meet (at infinity), and all three points of intersection should be
collinear. Thus, a full understanding of Pascal's theorem requires
knowledge of the projective plane, a geometric object described in
some detail in Section \ref{Section:Projective Plane}. 
 
Pascal's Theorem has an interesting converse, sometimes called the
Braikenridge--Maclaurin theorem after the two British mathematicians
William Braikenridge and Colin Maclaurin.

\begin{theorem}[Braikenridge--Maclaurin] If three lines meet three
other lines in nine points and if three of these points lie on a line
then the remaining six points lie on a conic. \label{Braikenridge-MaclaurinTheorem}
\end{theorem}

Braikenridge and Maclaurin seem to have arrived at the result independently, though
they knew each other and their correspondence includes a dispute over
priority. 

In 1848 the astronomer and mathematician August Ferdinand M\"obius
generalized the Braikenridge--Maclaurin Theorem. Suppose a polygon
with $4n+2$ sides is inscribed in a nondegenerate conic and we
determine $2n+1$ points by extending opposite edges until they
meet. If $2n$ of these $2n+1$ points of intersection lie on a line
then the last point also lies on the line. M\"obius's had developed a
system of coordinates for projective figures, but surprisingly his
proof relies on solid geometry. In Section \ref{Section:CBC} we prove
an extension of M\"obius's result, using the properties of projective
plane curves -- in particular, the Cayley-Bacharach Theorem. 

The Cayley--Bacharach Theorem is a wonderful result in projective
geometry. In its most basic form (sometimes called the 8 implies 9
Theorem) it says that if two cubic curves meet in 9 points then any
cubic through 8 of the nine points must also go through the ninth
point.  For the history and many equivalent versions of the
Cayley-Bacharach Theorem, see Eisenbud, Green and Harris's elegant
paper \cite{EGH}.  A strong version of the Cayley--Bacharach Theorem,
described in Section \ref{Section:CBC} , can be used to establish
another generalization of the Braikenridge--Maclaurin Theorem. 
The following existence theorem is well-known (see Kirwan's book on complex algebraic curves
\cite[Theorem 3.14]{Kirwan}) but we also claim a uniqueness result.  The
statement of Theorem \ref{construction} is inspired by the study of
hyperplane arrangements -- in this case, by collections of colored
lines in the plane. 

\begin{theorem} \label{construction} Suppose that $2k$ lines in the
  projective plane meet another line in $k$ triple points. Color the
  lines so that the line containing all the triple points is green and
  each of the $k$ collinear triple points has a red and a blue line
  passing through it.  Then there is a unique curve of degree $k-1$
  passing through the points where the red lines meet the blue lines
  off the green line.
\end{theorem}

When the red and blue lines have generic slopes, they meet in $k^2-k$
points off the green line. Since $\binom{k+1}{2}-1 = \frac{k^2+k-2}{2}$ points in general position determine a
unique curve of degree $k-1$ passing through the points,  it is
quite remarkable that the curve passes through all $k^2-k$
points of intersection off the green line. The Braikenridge--Maclaurin Theorem is just the instance $k=3$ of
Theorem \ref{construction}. The case where $k=4$ is illustrated in Figure
\ref{cubicexample}.  

\begin{figure}[h!t]
\begin{center}
\includegraphics[width=0.75\textwidth]{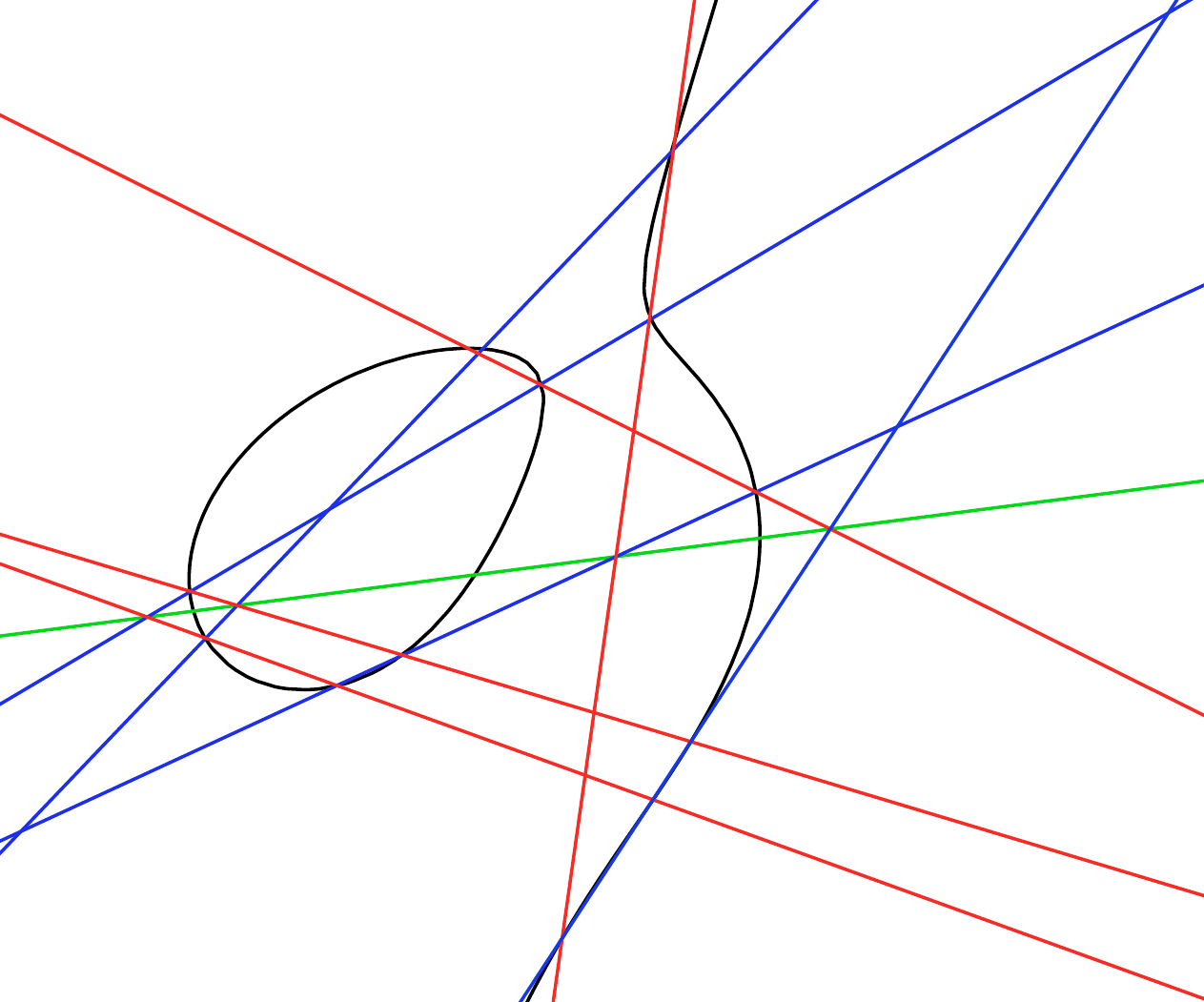}
\end{center}
\caption{An illustration of Theorem \ref{construction} when $k=4$. }
\label{cubicexample}
\end{figure}

We use the Cayley-Bacharach Theorem to prove Theorem
\ref{construction} in Section \ref{Section:CBC}. In Section
\ref{section:construction} we consider the kinds of curves produced by
the construction in Theorem \ref{construction}. For instance, we use
the group law on an elliptic curve to give a constructive argument
that, in a way that will be made precise, almost every degree-3 curve
arises in this manner. More generally, almost every degree-4 and
degree-5 curve arises in this manner. A simple dimension argument is
given to show that most curves of degree 6 or higher do not arise in
this manner. The proofs for degree 4 and 5 involve secant
varieties -- special geometric objects that have been quite popular
recently because of their applications to algorithmic complexity,
algebraic statistics, mathematical biology and quantum computing (see,
for example, Landsberg \cite{LandsbergBAMS, LandsbergPvsNP} ).

The last section contains some suggestions for further reading. As
well, Sections \ref{Section:Projective Plane} and \ref{Section:Fun}
contain amusing exercises that expand on the topic of the paper.

A paper generalizing a classical result in geometry cannot reference
all the relevant literature. One recent paper by Katz \cite{Katz}
is closely related to this work. His Mystic 2$d$-Gram \cite[Theorem
3.3]{Katz} gives a nice generalization of Pascal's Theorem; see
Exercise \ref{ex:fun}.\ref{ex:katz}.  He
also raises an interesting constructibility question: which curves can
be described as the unique curve passing through the $d^2-2d$ points
of intersection of $d$ red lines and $d$ blue lines that lie off a
conic through $2d$ intersection points?

\noindent {\em Acknowledgements: The author is grateful for
  conversations with my colleagues Mark Kidwell, Amy Ksir, Mark
  Meyerson, Thomas Paul, and Max Wakefield, and with my friend Keith
  Pardue. My college algebra professor Tony Geramita's work and
  conduct has been an inspiration to me. I have much to thank him for,
  but here I'll just note that he pointed me in the direction of some
  key ideas, including Terracini's lemma.  Many computations and
  insights were made possible using the excellent software packages
  Macaulay2, GeoGebra, Sage and Maple.  }

\end{section}
\begin{section}{Projective Geometry}
\label{Section:Projective Plane}

The general statement of Pascal's Theorem suggests that parallel lines
should meet in a point and that as we vary the pairs of parallel lines
the collection of such intersection points should lie on a line. This
is manifestly false in the usual Cartesian plane, but the plane can be
augmented by adding {\em points at infinity}, after which Pascal's Theorem
holds. The resulting projective plane $\P^2$ is a fascinating object
with many nice properties. 

One powerful model of the projective plane identifies points in $\P^2$
with lines through the origin in 3-dimensional space. To see how this relates to
our usual plane, consider the plane $z=1$ in 3-dimensional space as a
model for $\R^2$ and note that most lines through the origin meet this
plane. The line passing through $(x,y,1)$ is identified with the point
$(x,y) \in \R^2$. But what about the lines that don't meet this plane?
These are parallel to $z=1$ and pass through $(0,0,0)$ so they are
lines in the $xy$-plane. Each of these lines can be viewed as a
different point at infinity since they've been attached to our copy of
$\R^2$. 

In 1827 M\"obius developed a useful system of coordinates for points
in projective space \cite{M1827}, later extended by Grassmann. If we consider the punctured 3-space $\R^3
\setminus \{(0,0,0)\}$ and the equivalence relation 
$$ (x,y,z) \sim (\lambda x, \lambda y, \lambda z) \Leftrightarrow
\lambda \neq 0, $$
then each equivalence class corresponds to a line in $\R^3$ through the
origin. We denote the equivalence class of points on the line through
$(x,y,z)$ by $[x:y:z]$. This is a sensible notation since the
ratios between the coordinates determine the direction of the
line. Returning to our earlier model of $\P^2$, the points with $z
\neq 0$ correspond to points in our usual copy of $\R^2$, while the
points with $z=0$ correspond to points at infinity. 

If points in $\P^2$ correspond to lines through the origin, then what
do lines in $\P^2$ look like? Considering a line in $\R^2$ as
sitting in the plane $z=1$ we see that the points making up this line correspond
to lines through the origin that, together, form a plane. Any line in $\R^2$ can be
described by an equation of the form $ax+by+c = 0$; the reader should
check that this determines the plane $ax+by+cz = 0$. Thus, lines in
$\P^2$ correspond to dimension-2 subspaces of $\R^3$. In particular,
the line in $\P^2$ whose equation is $z=0$ is the line at infinity. 

We can also add points at infinity to $\R^n$ to create $n$-dimensional
projective space $\P^n$. Again, points in $\P^n$ can be identified with
1-dimensional subspaces of $\R^{n+1}$ and each point is denoted using
homogeneous coordinates $[x_0: x_1: \ldots : x_n]$. Similarly, we can
construct the complex projective spaces $\P^n_\mathbb{C}$, the points
of which can be identified with 1-dimensional complex subspaces of $\C^{n+1}$.

\begin{exercise} If this is the first time you've met projective
  space, you might try these, increasingly complicated, exercises. \label{ex:projective}
\begin{enumerate}
\item Show that if $ax+by+cz=0$ and $dx+ey+fz=0$ are two lines in
  $\P^2$ then they meet in a point $P = [g:h:i]$ given by the cross product, $$\langle g,h,i
  \rangle = \langle a,b,c \rangle \times \langle d,e,f \rangle.$$
\item Show that the line $ax+by+cz = 0$ in $\P^2$ consists of all the
  points of the form $[x:y:1]$ such that $ax+by+c=0$, together with a
  single point at infinity (the point $[b:-a:0]$). We say
  that the line $ax+by+cz=0$ is the {\em projectivization} of the line
  $ax+by+c=0$. 
\item Show that the projectivizations of two parallel lines
  $ax+by+c=0$ and $ax+by+d=0$ in $\R^2$ meet at a point at infinity. 
\item \label{projectivization} The projectivization of the hyperbola  $xy=1$ in $\R^2$ is the set of
  points in $\P^2$ that satisfy $xy - z^2 = 0$. Show that whether a point $[x:y:z]$ lies on the projectivization
  of the hyperbola or not is a well-defined property (i.e. the answer doesn't depend on
  which representative of the equivalence class $[x:y:z]$ we
  use). Where does the projectivization meet the line at infinity?
\item \label{S2} (a) By picking representatives of each equivalence class carefully,
  show that $\P^2$ can be put into 1-1 correspondence with the points
  on a sphere $S^2 \subset \R^3$, as long as we identify antipodal
  points, $(x,y,z) \sim (-x,-y,-z)$. \\
(b) Considering only the top half of the sphere, show that the points
in $P^2$ can be identified with points in a disk where antipodal points on the
boundary circle are identified.  \\
(c) Considering a thin band about the equator of the sphere from part (a),
show that $P^2$ can be constructed by sewing a M\"obius band onto the
boundary of a disk. How many times does the band twist around as we go
along the boundary of the disk? \\
(d) {\em Blowing up} is a common process in algebraic geometry. When
we blow up a surface at a point we replace the point with a
projectivization of its tangent space  (that is, the
space of lines through the base point in the tangent space). Show that
if we blow up a point on the sphere $S^2 \subset \R^3$ we get
$\P^2$. Show that if we blow up a second point we get a Klein bottle,
the surface obtained by sewing two M\"obius bands together along their
edges.   
\item (a) Viewing $\P^2$ as the set of lines in $\R^3$ through the
  origin, check that the map $\iota: \R^3 \setminus \{(0,0,0)\}
  \rightarrow \R^4$ given by
  $$\iota(x,y,z) =
  \left(\frac{x^2}{x^2+y^2+z^2},\frac{xy}{x^2+y^2+z^2},\frac{z^2}{x^2+y^2+z^2},\frac{(x+y)z}{x^2+y^2+z^2}
    \right)$$
  induces a well-defined embedding of $\P^2$ into $\R^4$. \\
  (b) Samuelson \cite{HS} gives an elegant argument to show that
  $\P^2$ cannot be embedded in $\R^3$. Maehara \cite{Why}, exploiting
  work of Conway and Gordon \cite{CG} and Sachs \cite{Sachs}, gives another simple
  argument for this fact based on the observation that any embedding of the complete graph $K_6$ in $\R^3$
  contains a pair of linked triangles. Complete the argument by drawing $K_6$ on $\P^2$ in such a way that no triangles are homotopically linked. \\
  (c) Amiya Mukherjee \cite{AM} showed that the complex projective
  plane $\P_\mathbb{C}^2$ can be smoothly embedded in $\R^7$. Research
  Question: Can $\P_\mathbb{C}^2$ be holomorphicly embedded in $\C^3$?
  Can $\P_\mathbb{C}^2$ be smoothly embedded in $\R^n$ with $n<7$?

\end{enumerate}
\end{exercise}

\begin{subsection}{The Power of Projective Space}

Projective space $\P^2$ enjoys many nice properties that Euclidean space
$\mathbb{R}^2$ lacks. Many theorems are much easier to state in projective space than
in Euclidean space. For instance, in Euclidean space any two lines
meet in either one point or in no points (in the case where the two
lines are parallel). By adding points at infinity to Euclidean space,
we've ensured that {\em any} two distinct lines meet in a point. This is just
the first of a whole sequence of results encapsulated in B\'ezout's
Theorem. 

Each curve $C$ in the projective plane can be described as the zero
set of a homogeneous polynomial $F(x,y,z)$: 
$$ C = \{ [x:y:z] \; : \; F(x,y,z) = 0 \}. $$
The polynomial needs to be homogeneous (all terms in the polynomial
have the same degree) in order for the curve to be well-defined (see
Exercise \ref{ex:projective}.\ref{projectivization}). It is
traditional to call degree-d homogeneous polynomials {\em degree-d
  forms}. The curve $C$ is said to be a degree-$d$ curve when
$F(x,y,z)$ is a degree-$d$ polynomial. We say that $C$ is an
irreducible curve when $F(x,y,z)$ is an irreducible polynomial.  When
$F(x,y,z)$ factors then the set $C$ is actually the union of several
component curves each determined by the vanishing of one of the
irreducible factors of $F(x,y,z)$. If $F(x,y,z) = G(x,y,z)H(x,y,z)$
then removing the component $G(x,y,z)=0$ leaves the {\em residual
  curve} $H(x,y,z)=0$.
 
\begin{theorem}[B\'ezout's Theorem] If $C_1$ and $C_2$ are curves of
  degrees $d_1$ and $d_2$ in the complex projective plane $\P^2_\C$
  sharing no common components then they meet in $d_1d_2$ points, counted appropriately. 
\end{theorem}

B\'ezout's Theorem requires that we work in {\em complex} projective
space: in $\P^2_\R$ two curves may not meet at all. For instance, the
line $y-2z=0$ misses the circle $x^2+y^2-z^2=0$ in $\P^2_\R$; the points
of intersection have complex coordinates. 

To say what it means to count appropriately requires a discussion of
intersection multiplicity. This can be defined in terms of the length
of certain modules \cite{Fulton}, but an intuitive
description will be sufficient for our purposes. When two curves meet
transversally at a point $P$ (there is no containment relation between
their tangent spaces) then $P$ counts as 1 point in B\'ezout's
Theorem. If the curves are tangent at $P$ or if
one curve has several branches passing through $P$ then
$P$ counts as multiple points. One way to determine how much the
point $P$ should count is to look at well-chosen families of curves $C_1(t)$ and
$C_2(t)$ so that $C_1(0)=C_1$ and $C_2(0)=C_2$ and to count how many
points in $C_1(t) \cap C_2(t)$ approach $P$ as $t$ goes to $0$. For
instance, the line $y=0$ meets the parabola $yz=x^2$ in one point
$P = [0:0:1]$. Letting $C_1(t)$ be the family of curves $y - t^2z = 0$
and letting $C_2(t)$ be the family consisting only of the parabola, we
find that $C_1(t) \cap C_2(t) = \{ [t:t^2:1], [-t:t^2:1]\}$ and so two
points converge to $P$ as $t$ goes to 0. In this case, $P$ counts with
multiplicity two. The reader interesting in testing their
understanding could check that the two concentric circles
$x^2+y^2-z^2=0$ and $x^2+y^2-4z^2$ meet in two points, each of
multiplicity two. More details can be found in Fulton's lecture notes
\cite[Chapter 1]{FultonCBMS}. 

It is traditional to call this result B\'ezout's Theorem because it
appeared in a widely-circulated and highly-praised book\footnote{Both the MathSciNet and
  Zentralblatt reviews of the English translation \cite{Bezout} are
  entertaining and well-worth reading. The assessment in the
  MathSciNet review is atypically colorful: ``This is not a book to be
  taken to the office, but to be left at home, and to be read on
  weekends, as a romance'', while the review in Zentralblatt Math
  calls it ``an immortal evergreen of astonishing actual relevance''. }. Indeed, in his
position as Examiner of the Guards of the Navy in France, \'Etienne
B\'ezout was responsible for creating new textbooks for teaching mathematics to
the students at the Naval Academy. However, Issac Newton
proved the result over 80 years before B\'ezout's book appeared!
Kirwan \cite{Kirwan} gives a nice proof of B\'ezout's Theorem. 

Higher projective spaces arise naturally when considering {\em moduli
  spaces} of curves in the projective plane. For instance, consider a
degree-$2$ curve $C$ given by the formula 
\begin{equation} \label{conic} a_0 x^2 + a_1 xy + a_2 xz + a_3 y^2 +
  a_4 yz + a_5 z^2 = 0. \end{equation} Multiplying the formula by a
nonzero constant gives the same curve, so the curve $C$ can be
identified with the point $[a_0:a_1:a_2:a_3:a_4:a_5]$ in $\P^5$. More
generally, letting $S=\C[x,y,z] = \oplus_{d \geq 0} S_d$ be the
polynomial ring in three variables, the degree-$d$ curves in $\P^2$
are identified with points in the projective space $\P(S_d)$, where we
identify polynomials if they are nonzero scalar multiples of one
another. A basis of the vector space $S_d$ is given by the
$D=\binom{d+2}{2}$ monomials of degree $d$ in three variables, so the
degree-$d$ curves in $\P^2$ are identified with points in the
projective space $\P(S_d) \cong \P^{D-1}$. Returning to the case of
degree-2 curves in $\P^2$, if we require $C$ to pass
through a given point, then the coefficients $a_0, \ldots,
a_5$ of $C$ must satisfy the linear equation produced by substituting
the coordinates of the point into (\ref{conic}). Now if we require $C$
to pass through $5$ points in $\P^2$ the coefficients must satisfy a
homogeneous system of 5 linear equations. If the points are in general
position (so that the resulting system has full rank), then the system
has a one-dimensional solution space and so there is just one curve
passing through all 5 points. In general, we expect a unique curve of
degree $d$ to pass through $D-1$ points in general position and we
expect no curves of degree $d$ to pass through $D$ points in general
position.

\end{subsection}
\end{section}
\begin{section}{A Generalization of the Braikenridge--Maclaurin Theorem}
\label{Section:CBC}

The following Theorem is a version of the Cayley-Bacharach Theorem
that was first proven by Michael Chasles. He used it to prove Pascal's
Mystic Hexagon Theorem, Theorem \ref{theorem:Pascal}. Because of its
content, the theorem is often called the $8 \Rightarrow 9$ Theorem.

\begin{theorem}[{$8 \Rightarrow 9$} Theorem] Let $C_1$ and $C_2$ be two plane cubic curves meeting
  in $9$ distinct points. Then any other cubic passing through any
  $8$ of the nine points must pass through the ninth point too. 
\end{theorem}

Inspired by Husem\"oller's book on Elliptic Curves \cite{Husemoeller},
Terry Tao recently gave a simple proof of the $8 \Rightarrow 9$ Theorem in his
blog\footnote{See Tao's July 15, 2011 post at {\tt
    terrytao.wordpress.com}.}. 

\begin{proof} (After Tao) The proof exploits the special position of
  the points in $C_1 \cap C_2  = \{P_1, \ldots, P_9\}$. Let $F_1=0$ and $F_2=0$ be the
  homogeneous equations of the curves $C_1$ and $C_2$. We will show
  that if $F_3$ is a cubic polynomial and $F_3(P_1) = \cdots = F_3(P_8) = 0$
  then $F_3(P_9)=0$. To do this it is enough to show that there are constants $a_1$
  and $a_2$ so that $F_3 = a_1F_1 + a_2F_2$ because then $F_3(P_9) =
  a_1F_1(P_9) + a_2F_2(P_9) = 0$. Aiming for a contradiction, suppose
  that $F_1$, $F_2$ and $F_3$ are linearly independent elements of
  $S_3$.

To start, no four of the points $P_1, \ldots, P_8$ can be collinear
otherwise $C_1$ intersects the line in $4 > (3)(1)$ points so the line
must be a component of $C_1$ by B\'ezout's Theorem. Similarly, the
line must be a component of $C_2$. But $C_1$ and $C_2$ only intersect
in 9 points so this cannot be the case. 

Now we show that there is a unique conic through any 5 of the points
$P_1, \ldots, P_8$. If two conics $Q_1$ and $Q_2$ were to pass through
5 of the points then by B\'ezout's Theorem they must share a
component. So either $Q_1=Q_2$ or both $Q_1$ and $Q_2$ are reducible
and share a common line. Since 4 of the points cannot lie on a line, the common
component must pass through no more than three of the five points. The
remaining two points determine the residual line precisely so $Q_1 =
Q_2$. 

Now we argue that in fact no three of the points $P_1, \ldots, P_8$
can be collinear. Aiming for a contradiction suppose that three of the
points lie on a line $L$ given by $H=0$ and the remaining 5 points lie on a conic
$C$. Since no 4 of the points $P_1, \ldots, P_8$ lie on a line, we
know that the 5 points on $C$ do not lie on $L$.  Pick constants
$b_1$, $b_2$ and $b_3$ so that the polynomial $F = b_1F_1 + b_2F_2 +
b_3F_3$ vanishes on a fourth point on $L$ and at another point $P \not\in L
\cup C$. Since the cubic $F=0$ meets the line $L$ in 4 points, $L$
must be a component of the curve $F=0$. But the residual curve given
by $F/H = 0$ is a conic going through 5 of the 8 points so it must be
$C$ itself. So $F=0$ is the curve $L \cup C$. But $F(P)=0$ by
construction and $P$ does not lie on $L \cup C$, producing the
contradiction. 

Now note that no conic can go through more than $6$ of the points
$P_1, \ldots, P_8$. B\'ezout's Theorem shows that the conic cannot go
through $7$ of the points, else  $C_1$ and $C_2$ would share a common
component. So suppose that a conic $C$ given by $G=0$ goes through 6
of the points. Then there is a line $L$ going through the remaining 2
points. The polynomial $G$ does not vanish at either of these 2 remaining points
because no $7$ of the points lie on a conic.  Pick constants $b_1$,
$b_2$ and $b_3$ so that the cubic $F = b_1F_1 + b_2F_2 + b_3F_3$
vanishes on a seventh point on the conic $C$ and another point $P
\not\in L \cup C$.  Now $F/G$ is a linear form that vanishes on the
two remaining points so $F/G=0$ must determine the line $L$. It
follows that $F=0$ is the curve $L \cup C$. Again, $F(P)=0$ by
construction and $P$ does not lie on $L \cup C$, producing the
contradiction. 

Now let $L$ be the line through $P_1$ and $P_2$ and let $C$ be the
conic through $P_3, \ldots, P_7$. From what we've proven above, $P_8$
does not lie on $L \cup C$.  Pick constants $c_1$,
$c_2$ and $c_3$ so that the cubic $F = c_1F_1 + c_2F_2 + c_3F_3$
vanishes on two more points $P$ and $P'$, both on $L$ but neither on $C$. Since
$F=0$ meets $L$ in four points, $F=0$ contains $L$ as a
component. Since $L$ cannot go through any of $P_3, \ldots, P_7$, the
residual curve mut be $C$, so $F=0$ is the curve $L \cup C$. But then $F(P_8)=0$ by
construction and $P_8$ does not lie on $L \cup C$, producing the
final contradiction. 
\end{proof}

Note that we proved slightly more: any cubic curve passing through $8$
of the nine points must be a linear combination of the two cubics
$C_1$ and $C_2$. 

Pascal's Mystic Hexagon Theorem, Theorem \ref{theorem:Pascal}, follows from an easy application of
the $8 \Rightarrow 9$ Theorem. Let $C_1$ be the cubic consisting of
the 3 lines formed by extending an edge of the hexagon and its two
adjacent neighbors. Let $C_2$ be the cubic consisting of the 3 lines
formed by extending the remaining, opposite, edges. $C_1$ and $C_2$
meet in 6 points on the conic $Q$ and in three points off the
conic. Let $L$ be the line through two of the three points of
intersection not on $Q$. Then $Q \cup L$ is a cubic curve through
$8$ of the $9$ points of $C_1 \cap C_2$. By the $8 \Rightarrow 9$
Theorem, $Q \cup L$ must contain the ninth point too. The point
cannot lie on $Q$ so it must lie on $L$. That is, the three points of
intersection not on $Q$ are collinear.

To prove the Braikenridge--Maclaurin Theorem, Theorem
\ref{Braikenridge-MaclaurinTheorem}, using the
Cayley-Bacharach Theorem, just observe that each collection of three
lines is a cubic curve (it is determined by the vanishing of a
degree-3 polynomial) and if three of the points lie on a line $L$ and
five of the remaining six points lie on a conic $C$ then $L \cup C$ is
a cubic curve passing through 8 of the nine points and so it must pass
through all nine points. However, the ninth point cannot lie on the
line $L$ if the original cubics meet only in points, otherwise $L$
would meet each of the original cubics in more than three points. So
the ninth point must be on the conic $C$.

A more powerful version of the Cayley-Bacharach Theorem can be found
in the last exercise of the Eisenbrick\footnote{An affectionate name 
  for David Eisenbud's excellent (and mammoth) tome on Commutative
  Algebra.} \cite[p. 554]{Eisenbrick} (also see
Eisenbud, Green and Harris \cite[Theorem CB5]{EGH}).  Before stating this result, we introduce some notation. 
Requiring a degree-$d$ curve in $\P^2$ to go through a point $p \in
\P^2$ imposes a non-trivial linear condition on the coefficients of
the defining equation of the curve. If a set $\Gamma$ of $\gamma$ points
imposes only $\lambda$ independent linear conditions on the
coefficients of a curve of degree $d$, then we say that {\em $\Gamma$ 
fails to impose $\gamma - \lambda$ independent linear conditions on
forms of degree $d$}. For example, 9 collinear points fail to impose 5
independent linear conditions on forms of degree 3 -- any cubic that
passes through 4 of the points must pass through them all. More
generally, any set of $k$ collinear points fails to impose $k-(d+1)$
conditions on forms of degree $d \leq k-1$. 

\begin{theorem}[Cayley-Bacharach] \label{CB}
Suppose that two curves of degrees $d_1$ and $d_2$ meet in a finite
collection of points $\Gamma \subset \P^2$. Partition $\Gamma$ into
disjoint subsets: $\Gamma = \Gamma' \cup \Gamma''$ and set $s = d_1+d_2
-3$. If $d \leq s$ is a non-negative integer then the dimension of the space of forms of degree
$d$ vanishing on $\Gamma'$, modulo those vanishing on $\Gamma$, is
equal to the failure of $\Gamma''$ to impose independent conditions on forms of degree
$s-d$.  
\end{theorem}

We restate our generalization of the Braikenridge-Maclaurin Theorem
and give a proof using the Cayley-Bacharach Theorem. Kirwan
\cite[Theorem 3.14]{Kirwan} gives a simple proof for the existence
part of the Theorem that is well-worth examining. 

\setcounter{theorem}{3}
\begin{theorem} Suppose that $2k$ lines in the
  projective plane meet another line in $k$ triple points. Color the
  lines so that the line containing all the triple points is green and
  each of the $k$ collinear triple points has a red and a blue line
  passing through it.  Then there is a unique curve of degree $k-1$
  passing through the points where the red lines meet the blue lines
  (off the green line). 
\end{theorem}
\setcounter{theorem}{8}

\begin{proof}  The cases $k=1$ and $k=2$ are
  trivial so assume $k \geq 3$.  Suppose that the red lines are cut
  out by the forms $L_1,\ldots,L_k$, the blue lines are cut out by the
  forms $M_1, \ldots, M_k$ and the green line is cut out by the form
  $G$. Let $\Gamma$ be the points of intersection of the two degree
  $k$ forms $L_1 \cdots L_k$ and $M_1 \cdots M_k$. Note that there are
  no degree-$(k-1)$ curves that pass through all the points of $\Gamma$:
  any such curve meets the line $M_i = 0$ in $k$ points so each $M_i$
  divides the equation of the curve, leading to a contradiction on the
  degree of the defining equation. Now let $\Gamma'$ be the points of
  $\Gamma$ that lie off the green line and $\Gamma''$ the points of
  $\Gamma$ on the green line. The $k$ collinear points in $\Gamma''$
  impose $k-1$ independent conditions on forms of degree $k-2$. So
  $\Gamma''$ fails to impose $k-(k-1) =1$ condition on forms of degree
  $k-2$.  Because there are no curves of degree $k-1$ going through
  all of $\Gamma$, the Cayley-Bacharach Theorem says that the
  dimension of the space of forms of degree $k-1$ vanishing on
  $\Gamma'$ is equal to the failure of $\Gamma''$ to impose
  independent conditions on forms of degree $k+k-3-(k-1)=k-2$, which
  is one. So up to scaling, there is a unique equation of degree $k-1$
  passing through the points of $\Gamma$ off the green line.
\end{proof}

In the generic case, just one red and one blue line pass
through each point of intersection, and the curve $S$ passes through
all $k^2-k$ points of intersection between the red lines and the blue
lines that do not lie on the green line. If we are not in the generic
case the uniqueness claim needs further interpretation. For those that
know about intersection multiplicity, the curve $S$ is the unique
curve whose intersection multiplicity with the union of the red lines at the
points of intersection off the green line equals the intersection
multiplicity of the union of the blue lines with the union of the red
lines at those same points (and we can replace red with blue in this
statement). 

M\"obius \cite{Mobius} also generalized the Braikenridge-Maclaurin Theorem,
but in a different direction.  Suppose a polygon with $4n+2$ sides is
inscribed in a irreducible conic and we determine $2n+1$ points by
extending opposite edges until they meet. If $2n$ of these $2n+1$
points of intersection lie on a line then the last point also lies on
the line. Using the Cayley-Bacharach Theorem allows us to extend
M\"obius's result, relaxing the constraint on the number of sides of
the polygon.

\begin{theorem}
  Suppose that a polygon with $2k$ sides is inscribed in an irreducible
  conic. Working around the perimeter of the polygon, color the edges
  alternately red and blue. Extending the edges to lines, consider the
  $k^2-2k$ points of intersection of the red and blue lines, omiting
  the original $2k$ vertices of the polygon. If $k-1$ of these points
  lie on a green line, then in fact another of these points lies on
  the green line as well. \label{Mobiusextension}
\end{theorem}

The theorem is illustrated when $k=4$  in Figure \ref{Mobius8}:
both the green and purple lines contain 3 of the 8 points off the
conic, so they must each contain a fourth such point too. 

\begin{figure}[h!t]
\begin{center}
\includegraphics[width=0.75\textwidth]{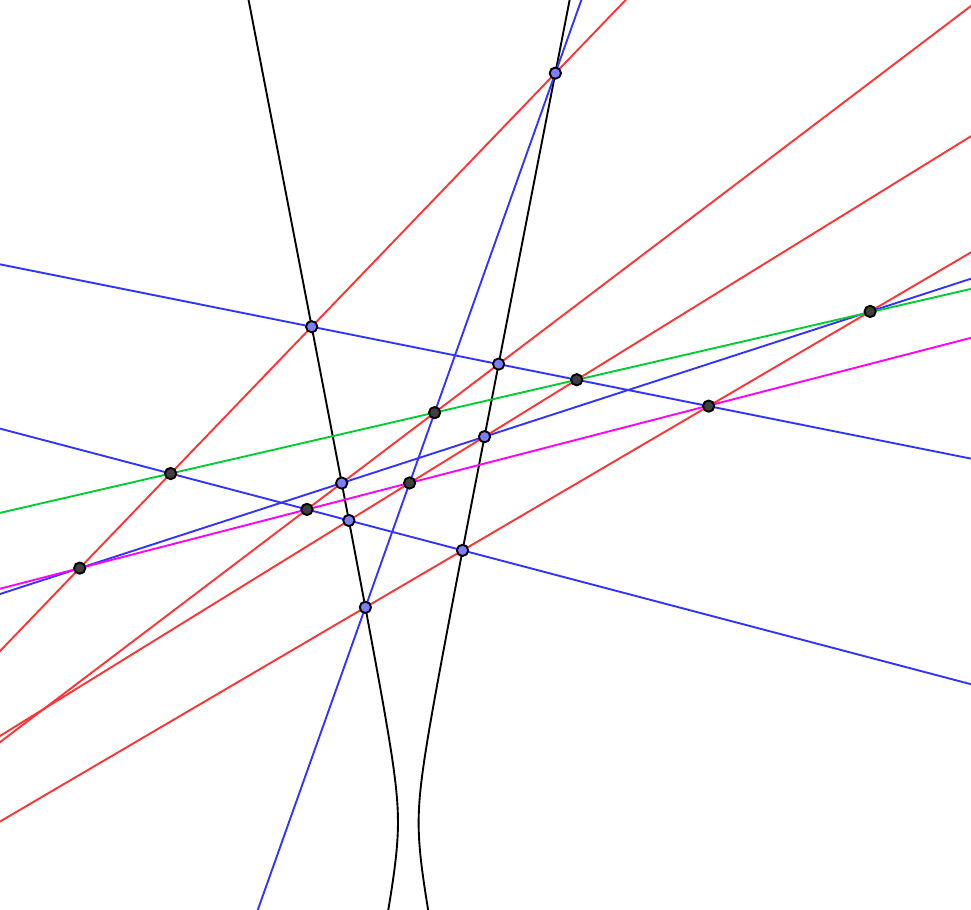}
\end{center}
\caption{An illustration of Theorem \ref{Mobiusextension} when $k=4$. }
\label{Mobius8}
\end{figure}
 
\begin{proof}[Proof of Theorem \ref{Mobiusextension}.]
  First compute the dimension of the degree-$(k-2)$ curves that go
  through all the points off the conic. Since there are no
  degree-$(k-2)$ curves through all the $k^2$ points of intersection
  of the extended edges, the Cayley-Bacharach Theorem gives that this
  dimension equals the failure of $2k$ points that lie on a conic to
  impose independent conditions on curves of degree $k+k-3-(k-2)=
  k-1$. Because the conic is irreducible, these degree-$(k-1)$ curves
  must contain the conic as a component. So the failure is $2k$ minus
  the difference between the dimension of the space of degree $k-1$
  curves in $\P^2$ and the dimension of the space of degree $k-3$
  curves in $\P^2$. The failure is thus $$2k - \left[ \binom{k+1}{2} -
    \binom{k-1}{2} \right] = 1. $$ So, up to scaling, there is a
  unique curve of degree $k-2$ through all the points off the
  conic. Since this curve meets the green line in at least $k-1$
  points, B\'ezout's Theorem shows that it must contain the line as a
  component. Taking the union of the residual curve (of degree $k-3$)
  with the conic gives a curve of degree $k-1$ through all the points
  on both the red and blue lines that do not lie on the green
  line. Now the Cayley-Bacharach Theorem says that the dimension of
  all degree-$(k-1)$ curves through all the points off the green line
  equals the failure of the points on the green line to impose
  independent conditions on forms of degree $k-2$. There are at least
  $k-1$ points on the green line, so the failure equals $$ \text{(\#
    points on the line)} - \left[ \binom{k}{2} - \binom{k-1}{2}
  \right] = \text{(\# points on the line)} - (k-1). $$ But there is
  such a curve so this number must be at least one, in which case the
  number of points on the green line must be at least $k$. Of course,
  the number of points on the green line is bounded by the number of
  points on the intersection of the green line with the $k$ red lines
  so there are precisely $k$ intersection points on the green
  line. This shows that the last point must also lie on the green line
  and establishes M\"obius's result.
\end{proof}

\end{section}
\begin{section}{Constructible curves}
\label{section:construction}

Let's take a constructive view of Theorem \ref{construction}, our
extension of the Braikenridge-Maclaurin Theorem. Say that
a curve $X$ of degree $d$ is {\em constructible} if there exist $d+1$ red
lines $\ell_1, \ldots, \ell_{d+1}$ and $d+1$ blue lines $L_1, \ldots,
L_{d+1}$ so that the $d+1$ points $\{\ell_i \cap L_i: 1 \leq i \leq
d+1\}$ are collinear and the other $d(d+1)$ points $\{\ell_i \cap L_j:
i \neq j\}$ lie on $X$. 

We turn to the question of which curves are constructible.  In
particular, we aim to show that for a certain range of degrees $d$
almost all curves of degree $d$ are constructible and for degrees
outside of this range, almost no curves of degree $d$ are
constructible. One way to make such statements precise is to introduce
the Zariski topology on projective space. 

The Zariski topology is the coarsest topology that makes polynomial
maps from $\P^m$ to $\P^n$ continuous. More concretely, every
homogeneous polynomial $F$ in $n+1$ variables determines a closed set
in $\P^n$ $$\V(F) =
\{P \in \P^n \; : \; F(P)=0\},$$ 
and every closed set is built up by taking finite unions and arbitrary
intersections of such sets. Closed sets in the Zariski topology are
called varieties. The nonempty open sets in this topology
are dense: their complement is contained in a set of the form
$\V(F)$. 

We'll say that {\em the construction is dense for degree-$d$ curves}
if there is a nonempty Zariski-open  set of degree-$d$ curves $U$ such that
each $X \in U$ is constructible. 

\begin{question}
For which degrees is the construction dense? 
\end{question}
 
The construction is clearly dense for degree-1 curves
(lines). Pascal's Theorem, Theorem \ref{theorem:Pascal}, shows that the construction is dense
for degree-2 curves. 

We give a simple argument to show that the construction cannot be
dense if $d \geq 6$. Consider the number of parameters that can be
used to define an arrangement of $2d+3$ lines so that there are $d+1$
triple points on one of the lines. Two parameters are needed to define
the green line and then we need $d+1$ parameters to determine the
triple points and $2(d+1)$ parameters to choose the slopes of pairs of
lines through these points. So a $\left( 3d+5\right)$-dimensional
space parameterizes the line arrangements. The space of degree-$d$
curves is parameterized by a $\left( \binom{d+2}{2}-1\right)$
projective space. Since $$\binom{d+2}{2}-1 = \frac{d^2+3d}{2} >
3d+5, $$ when $d \geq 6$, it is impossible for the line arrangements
to parameterize a nonempty Zariski-open set of dimension
$\binom{d+2}{2}-1$ when $d \geq 6$.

Using the group law on elliptic curves allows
us to show that the construction is dense for degree-3 curves.

\begin{theorem} \label{constructibledeg3} The construction is dense for degree-3
  curves. \end{theorem} 

\begin{proof} The set of smooth plane curves of degree 3 is a nonempty
  Zariski-open set in the space $\P^9$ parameterizing all degree-3
  curves \cite[Theorem 2 in Section II.6.2]{Shaf1}. Such curves are called elliptic curves and their
  points form a group:  three distinct points add to the identity element in the
  elliptic curve group if and only if they are collinear\footnote{If
    two of the points are the same then the line must also be tangent
    to $X$ at this point, while if all three points are the same then
    the tangent line to $X$ at the point must intersect $X$ with
    multiplicity 3.}.  Given an elliptic curve $X$ we pick 5 points,
  $p_1$, \ldots, $p_5$ on the curve, no three of which are
  collinear. We will construct the red, blue and green lines; the
  reader may wish to refer to the schematic diagram in Figure
  \ref{schematic} as the construction proceeds. 

\begin{figure}[h!t]
\begin{center}
\includegraphics[width=0.75\textwidth]{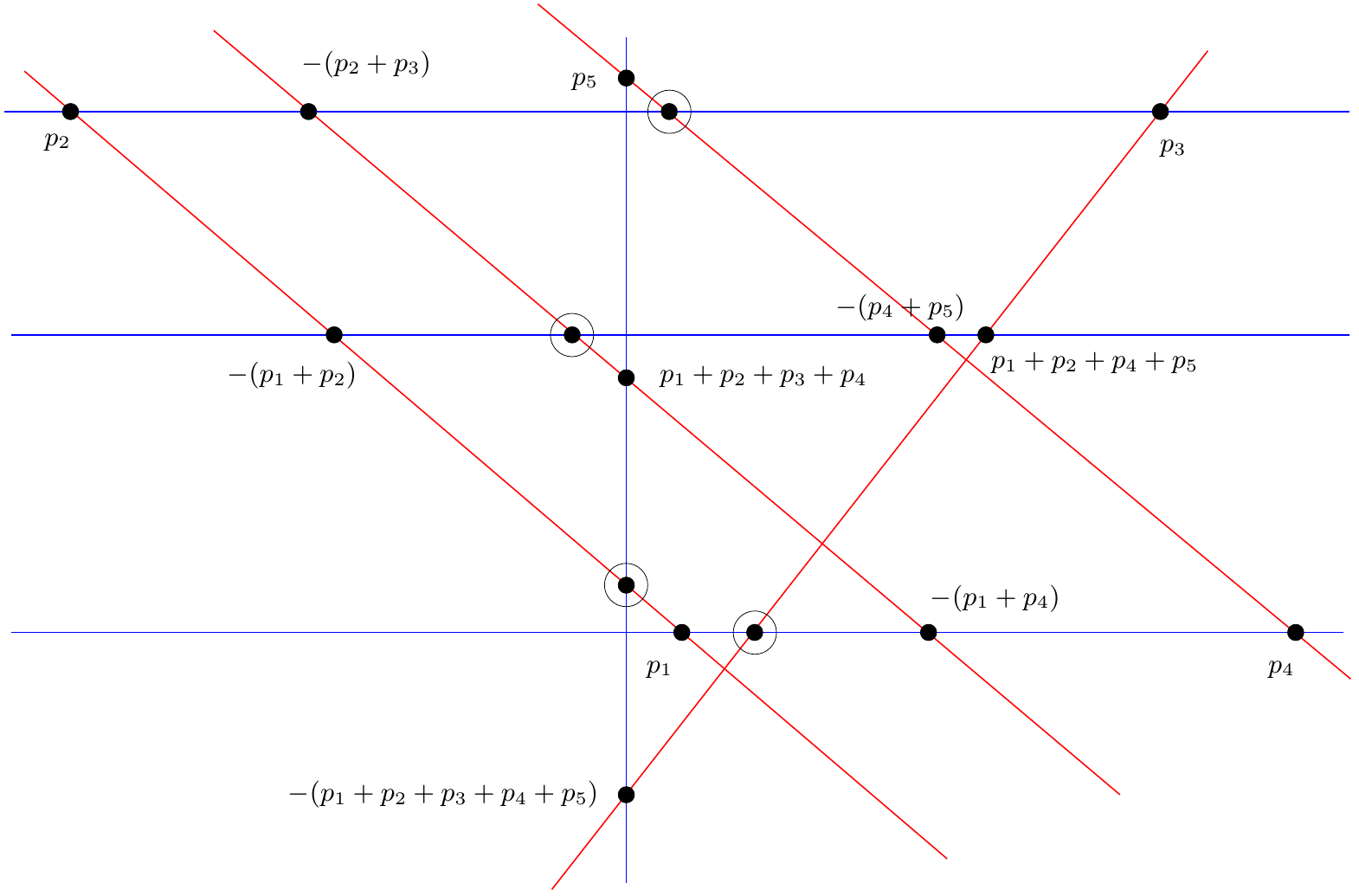}
\end{center}
\caption{A Schematic Illustration of the Construction in the Proof of
  Theorem \ref{constructibledeg3}.  }
\label{schematic}
\end{figure}

We draw a red line
  connecting points $p_1$ and $p_2$, meeting $X$ in the third point
  $-(p_1+p_2)$.  A blue line joining $p_1$ and $p_4$ meets $X$
  at $-(p_1+p_4)$ and a blue line joining $p_2$ and $p_3$ meets $X$ at
  $-(p_2+p_3)$. A red line joining $p_4$ and $p_5$ meets $X$ at
  $-(p_4+p_5)$. A red line joins $-(p_1+p_4)$ and $-(p_2+p_3)$ and
  meets $X$ in the point $p_1+p_2+p_3+p_4$. A blue line joins
  $p_1+p_2+p_3+p_4$ to $p_5$, meeting $X$ in
  $-(p_1+p_2+p_3+p_4+p_5)$. A red line joins $-(p_1+p_2+p_3+p_4+p_5)$
  to $p_3$, meeting $X$ in the point $p_1+p_2+p_4+p_5$. A blue line
  through $-(p_1+p_2)$ and $-(p_4+p_5)$ also hits $X$ at
  $p_1+p_2+p_4+p_5$. The four red lines meet the four blue lines in 16
  points, 12 of which lie on the elliptic curve $X$. We will prove
  that the other 4 points, circled in the schematic Figure
  \ref{schematic} are collinear (lying on the green line) using the
  Cayley-Bacharach Theorem. Indeed, let $\Gamma$ be the 16 points
  where the red lines meet the blue lines and let $\Gamma''$ be the 12
  points lying on the cubic $X$. Let $\Gamma' = \Gamma \setminus
  \Gamma''$ be the residual set of the four circled points. Since
  there are no degree-1 curves vanishing on the 16 points of $\Gamma$, the
  Cayley-Bacharach Theorem says that the dimension of the space of
  degree-1 curves vanishing on all four points of $\Gamma'$ equals the
  failure of $\Gamma''$ to impose independent conditions on curves of
  degree $4+4-3-1=4$. The failure equals $12$ minus the codimension of
  the degree-4 forms vanishing on $\Gamma''$ in the space of all
  degree-4 forms. This is equal to three less than the dimension of
  the vector space of degree-4 forms vanishing on $\Gamma''$:
$$ \begin{aligned} & 12 - \left[ \binom{6}{2} - \text{dim degree-4
      forms vanishing on } \Gamma^{\prime \prime}    \right] \\
= \; \; & \text{dim degree-4
      forms vanishing on } \Gamma^{\prime \prime} - 3. 
\end{aligned}
$$
 Now any linear form
times the equation of the cubic $X$ gives a degree-4 form vanishing on
$\Gamma''$ so the degree-4 forms vanishing on $\Gamma''$ is a vector
space of dimension at least 3. However, the defining ideal of the four
red lines also vanishes on $\Gamma''$, so in fact the dimension is at
least 4. It follows that the failure is at least 1 so the four points
in $\Gamma'$ are collinear.  So the construction produces
all elliptic curves and is dense in degree 3.
\end{proof}



To establish that the construction is dense in degrees 4 and 5, we use Terracini's
beautiful lemma about secant varieties {\cite[Lemma 3.1]{CCG}} (stated below in a restricted
form, though it holds for higher secant varieties too). If $X$ is a subvariety of
$\P^n$ and $p_1 \neq p_2$ are two points on $X$ then the line joining
$p_1$ to $p_2$ is a secant line to $X$. The secant line variety $\text{Sec}(X)$ is the Zariski-closure
of the variety of points $q \in \P^n$ that lie on a secant line to
$X$. 

\begin{lemma}[Terracini's Lemma] 
  Let $p$ be a generic point on $\text{Sec}(X) \subset \P^n$, lying on
  the secant line joining the two points $p_1 \neq p_2$ of $X$. Then
  $T_p(\text{Sec}(X))$, the (projectivized) tangent space to
  $\text{Sec}(X)$ at $p$, is  $\langle
  T_{p_1}(X), \; T_{p_2}(X)\rangle$, the projectivization of the  linear span of the two vector
  spaces $T_{p_1}(X)$ and $T_{p_2}(X)$. In particular,
$$ \text{dim}\;  \text{Sec}(X)  = \text{dim}\;  \langle T_{p_1}(X), \; T_{p_2}(X)\rangle.$$
\end{lemma}

We apply this in the setting where $X = \mathbb{X}_{1^5}$, the
variety of completely reducible forms of degree
$5$ on $\P^2$. Letting $S = \C[x,y,z] = \oplus_{d\geq 0} S_d$, $\mathbb{X}_{1^5}$ is a
subvariety of the parameter space $\P(S_5)$ of all degree-$5$ curves: 
$$\mathbb{X}_{1^5} = \{ [F_1\cdots F_5]: \; \text{each} \; F_i \in S_1\}.$$ 
Fortunately, Carlini, Chiantini and Geramita recently
described the tangent space to $\mathbb{X}_{1^d}$. 

\begin{lemma}[{\cite[Proposition 3.2]{CCG}}]
The tangent space to a point  $p = [F_1\cdots F_d]$ in $\mathbb{X}_{1^d} \subset
\P(S_d)$ is the projectivization of the degree-$d$ part of the ideal  $$I_p = \langle
G_1, \ldots, G_d \rangle,$$ where $G_i = (F_1F_2\ldots F_d)/F_i$. 
\end{lemma}

If the forms $F_1, \ldots, F_5$ are distinct then $(I_p)_5$ has
dimension $11$. To see this, note that $\text{dim}\;
(I_p)_5 = (\text{dim}S_1)(\text{dim} (I_p)_{4}) - \text{dim}
(\text{Syz}(G_1,\ldots, G_5)_1) = 3(5) - \text{dim}
(\text{Syz}(G_1,\ldots, G_5)_1)$, where $$\text{Syz}(G_1,\ldots,
G_5)_1 = \{ (L_1,\ldots,L_5) \in (S_1)^5: \; L_1G_1 + \cdots +  L_5G_5 =
0  \} $$ 
is the degree-$1$ part of the syzygy module. However, if $L_1G_1 +
\cdots + L_5G_5 = 0$ then $L_1G_1 = -(L_2G_2+ \cdots +L_5G_5)$. Since
$F_1$ divides each of the terms on the right-hand side of this
equality, $F_1$ must divide $L_1G_1$. But $F_1$ does not divide $G_1$ so $F_1 |
L_1$. Similarly, $F_i | L_i$ for $i = 1, \ldots, 5$. It follows that
the only degree-$1$ syzygies are generated by the $4$ linearly
independent syzygies $F_1e_1 - F_ie_i$ ($i=2, \ldots, 5$). As a
result, $T_p(X_{1^5})$ is the projectivization of an $11$-dimensional
vector space. 

Terracini's Lemma shows that if $p$ is a generic point on the line
between $p_1$ and $p_2$, 
$T_{p}(\text{Sec} \mathbb{X}_{1^5})$ is the projectivization
of the linear span of the vector spaces $(I_{p_1})_5$ and $(I_{p_2})_5
$.  This span has dimension $2(11) - \text{dim}\; (I_{p_1} \cap
I_{p_2})_5$. If $I$ is an ideal in $S$, let $\V(I)$ denote the set of points $P
\in \P^2$ such that $F(P)=0$ for all $F \in I$. Now if $p_1 = [F_{11}\cdots F_{15}]$ and $p_2 = [F_{21}
\cdots F_{25}]$ and if all the lines $\V(F_{ij})$ 
are distinct, then $\mathbb{V}(I_{p_1} \cap I_{p_2}) =
\mathbb{V}(I_{p_1}) \cup \mathbb{V}(I_{p_2})$ is a collection of 20
points (counted with multiplicities): 10 of the points are given
by the intersections of the $\binom{5}{2}$ pairs of lines
$F_{1i}(x,y,z)=0$ and 10 of the points are given
by the intersections of the $\binom{5}{2}$ pairs of lines
$F_{2i}(x,y,z)=0$. A polynomial in $(I_{p_1} \cap I_{p_2})_5$ is a curve that goes through these 20 points. If the
20 points were in general position, we would expect only 1 curve to go
through all 20 points and so $\text{dim}\; (I_{p_1} \cap I_{p_2})_5 =
1$. However, the 20 points are in special position -- for example,
many collections of four of the points are collinear -- so we cannot
trust our intuition blindly. Note that an element $H$ of $(I_{p_1} \cap
I_{p_2})_5$ corresponds to a solution to a system of equations $A{\bf
  v} = {\bf 0}$ where $A$ is a 20 $\times$ 21 matrix whose columns
correspond to the monomials of $S_5$ and whose rows correspond to the
20 points in $\V(I_{p_1}) \cup \V(I_{p_2})$. The 21 entries of a
solution ${\bf v}$ are the coefficients of $H(x,y,z)$. The entries of
$A$ in the column corresponding to a given monomial are
obtained by plugging in the coordinates of a point into the
monomial. The entries of the point at the intersection of $F_{ij}(x,y,z)=a_{ij}x +
b_{ij}y + c_{ij}z = 0$ and $F_{ik}(x,y,z) = a_{ik}x +
b_{ik}y + c_{ik}z = 0$ are given by the cross product $\langle a_{ij},
b_{ij}, c_{ij} \rangle \times \langle a_{ik}, b_{ik}, c_{ik}
\rangle$. So the entries in the matrix $A$ are degree-$10$ polynomials
in the coefficients of the $F_{ij}$. This matrix will have full rank
unless all the $20\times 20$ minors are zero. This means that the
matrix has full rank off the closed set where all the maximal minors
vanish. So the matrix has full rank (and dim $ (I_{p_1} \cap
I_{p_2})_5 = 1$) on an open set. To show that this open set is dense
we just need to show that it is nonempty by exhibiting an example. 

The {\tt hilbertFunction} command in Macaulay2
\cite{M2} can be used to compute the dimension of $(I_{p_1}\cap
I_{p_2})_5$. Checking a randomly selected example shows that
$\text{dim}\; (I_{p_1} \cap I_{p_2})_5 = 1$ generically and so for a
generic point $p$ of $\text{Sec} (X_{1^5})$, we see that the tangent space at
$p$ is the projectivization of a $21$-dimensional space; that is,
$\text{Sec} (X_{1^5}) \subset \P^{20}$ is a $20$-dimensional projective
variety and so $\text{Sec} (X_{1^5}) = \P^{20}$. We're now ready to tackle
the constructibility question for curves of degrees 4 and 5.  

\begin{theorem} The construction is dense for curves of degree 4. \label{construct4}
\end{theorem}

\begin{proof} We'll show that there is a dense open subset of
  constructible {\em irreducible} curves of degree 4. The irreducible
  curves of degree 4 are themselves dense and open in the set of all
  curves of degree 4; see Shafarevich's book \cite[Section 5.2]{Shaf1}
  for details. First we note that the set of degree-$5$ forms $Z \in
  \P(S_5)$ such that there exist $p_1 \in X_{1^5}$ and $p_2 \in
  X_{1^5}$ so that $Z$ is a linear combination of $p_1$ and $p_2$ and
  the 10 lines in $\V(p_1) \cup \V(p_2)$ are {\em not} distinct is
  contained in a 19-dimensional subvariety $W$ of $\P(S_5)$. The
  dimension count is easy: there are two parameters for each of the 9
  (possibly) distinct lines determining $p_1$ and $p_2$ and 1
  parameter to reflect where $Z$ lies on the line joining $p_1$ and
  $p_2$. Now fix a linear form $L$ and consider the map $\phi_L:
  \P(S_4) \rightarrow \P(S_5)$ given by multiplication by $L$. The
  inverse image $\phi_L^{-1}(W^c)$ of the complement of $W$ is open in
  $\P(S_4)$.  Given an {\em irreducible} degree-4 form $F$ in this
  open set, there exist $p_1$ and $p_2$ in $X_{1^5}$ so that $FL$ is a
  linear combination of $p_1$ and $p_2$ and the 10 lines in $p_1$ and
  $p_2$ are distinct. Then $\V(FL)$ contains the 25 points of
  intersection between the lines in $p_1$ and the lines in $p_2$. We
  claim that 5 of the 25 points in $\V(p_1) \cap \V(p_2)$ lie on
  $\V(L)$ and the remaining 20 points lie on $\V(F)$. If more than 5
  points lie on $\V(L)$ then B\'ezout's Theorem shows that $L$ must
  divide $p_1$. Similarly, $L$ must divide $p_2$. This is impossible
  because the 10 lines in $\V(p_1)$ and $\V(p_2)$ are distinct.
  Similarly, if $\V(F)$ goes through more than 20 points of $\V(p_1)
  \cap \V(p_2)$ then $F$ and $p_1$ must have a nontrivial common
  divisor. But $F$ is irreducible so this cannot occur. It follows
  that $\V(F)$ is constructible (the red lines are the lines in $\V(p_1)$, the
  blue lines are the lines in $\V(p_2)$ and the green line is the line
  $\V(L)$). 

It follows that an open set of irreducible degree-$4$ curves is
constructible. We give an example to show that this open set is nonempty.
Take the green line to be $y=0$, the red lines
to be $x+2z=0$, $x+z=0$, $x=0$, $x-z=0$, and $x-2z=0$, and the blue
lines to be $x-y+2z=0$, $x-y+z=0$, $x-y=0$, $x-y-z=0$, and
$x-y-2z=0$. The red lines intersect the blue lines in 20 distinct
points off the green line and the polynomial 
$$        	
5 x^{4} - 10 x^{3} y + 10 x^{2} y^{2} - 5 x y^{3} + y^{4} - 15 x^{2} +
15 x y - 5 y^{2} + 4$$
vanishes on each of the 20 points. You can, for example, dehomogenize
the polynomial (set $z =1$) and use Maple's
{\tt evala(AFactor($\cdot$))} command to check that the polynomial
is irreducible. 
\end{proof}

\begin{theorem} The construction is dense for curves of degree 5. 
\label{construct5}
\end{theorem}

\begin{proof}
First we note that $X_{1,V}$, the subvariety of $\P(S_6)$ consisting
of degree 6 forms that factor into a linear form times an irreducible
degree-$5$ form, is in fact a subvariety of $\Sec(X_{1^6})$. If $L$ is a
linear form and $Q$ is an irreducible degree-$5$ form then $Q \in
\P(S_5) = \Sec(X_{1^5})$ so there are completely reducible forms $p_1$
and $p_2$ of degree $5$ so that $Q$ is a linear combination of $p_1$
and $p_2$. It follows that the form $LQ \in X_{1,V}$ is a linear
combination of $Lp_1$ and $Lp_2$ so $X_{1,V} \subseteq \Sec(X_{1^6})$. 
Moreover, $X_{1,V}$ is a closed set in $\P(S_6)$ since it is the image
of the regular map $\P(S_1) \times V \rightarrow \P(S_6)$, where the
map is given by multiplication and $V$ is projectivization of the
irreducible forms of degree 5. This shows that $X_{1,V}$ is a subvariety of
$\Sec(X_{1^6})$. 

Now pick $Z \in V$ constructible so that $L$ is the defining equation
of the green line (for some set of distinct red and blue lines). For
example, we can take the green line to be $y=0$, the red lines
to be $x+3z=0$, $x+2z=0$, $x+z=0$, $x=0$, $x-z=0$ and $x-2z=0$, and the
blue lines to be $x-y+3z=0$, $x-y+2z=0$, $x-y+z=0$, $5x-y=0$,
$5x-y-5z=0$ and $5x-y-10z=0$. The red lines intersect the blue lines
in 30 distinct points off the green line and the irreducible
polynomial $$\begin{aligned} 450 x^{5} & - 615 x^{4} y + 396 x^{3} y^{2} - 123 x^{2} y^{3} + 18 x y^{4}
- y^{5} + 675 x^{4} z - 150 x^{3} y z \\ & - 234 x^{2} y^{2} z 
 + 93 x y^{3} z
- 9 y^{4} z - 2400 x^{3} z^{2} + 2250 x^{2} y z^{2} - 504 x y^{2} z^{2}
\\ & + 29 y^{3} z^{2} - 2025 x^{2} z^{3} - 375 x y z^{3} + 141 y^{2} z^{3} +
2400 x z^{4} - 460 y z^{4} + 300 z^{5} \end{aligned}$$ 
vanishes on each of the 30 points. Fixing $L$,
consider the map $\phi_L: V \rightarrow \Sec(X_{1^6})$ given by
sending an irreducible degree-5 form $F$ to $FL \in X_{1,V} \subset
\Sec(X_{1^6})$. The set of points in $\Sec(X_{1^6})$ that lie on the
line connecting completely reducible forms $p_1$ and $p_2$ where the
12 lines forming $\V(p_1) \cup \V(p_2)$ are {\em not} distinct is a
closed set $W$ of dimension no larger than 23. We leave it to the
reader to check that $\Sec(X_{1^6})$ has dimension 25; the proof is
similar to the argument given above that $\Sec(X_{1^5})$ has dimension
20. It follows that the inverse image $\phi_L^{-1}(W^c)$ of the
complement of $W$ is open in $V$. Since $\phi_L(Z) \not\in W$, the
open set is nonempty. Now if $F$ is a degree-5 irreducible form with $FL
\not\in W$,  there exist $p_1$ and $p_2$ in $X_{1^6}$ so
  that $FL$ is a linear combination of
  $p_1$ and $p_2$ and the 12 lines in $p_1$ and $p_2$ are
  distinct. Then $\V(FL)$ contains the 36 points of intersection
  between the lines in $p_1$ and the lines in $p_2$. Now, as in the
  proof of Theorem \ref{construct4}, B\'ezout's Theorem shows that 6
  of the 36 points in $\V(p_1) \cap \V(p_2)$ lie on $\V(L)$ and the
  remaining 30 points lie on $\V(F)$. This allows us to use the lines
  in $V(p_1)$ as our red lines, the lines in $\V(p_2)$ as the blue
  lines and the line $\V(L)$ as our green line to construct the curve
  $\V(F)$. 

We've shown that a nonempty open subset of the irreducible degree-5
curves consists of constructible curves. The result follows because
the collection of irreducible degree-5 curves form an open set in the
parameter space $\P(S_5)$. 
\end{proof}

We have not provided an example of a curve of degree less than 6 that
is {\em not} constructible. It may be that the set of constructible
curves is Zariski-closed. In this case, {\em every} curve of degree less
than 6 would be constructible because projective spaces are connected
in the Zariski-topology: the only sets in projective space that are
both open and closed are the empty set and the whole space.  

\end{section}
\begin{section}{Further Reading and Exercises}
\label{Section:Fun}

Pappus's Theorem inspired a lot of amazing mathematics. The first
chapter of a fascinating new book by Richter-Gebert \cite{RG} describes the connections between
Pappus's Theorem and many areas of mathematics, including cross-ratios
and the Grassmann-Pl\"ucker relations among determinants. 

The history and implications of the Cayley-Bacharach Theorem is
carefully considered in Eisenbud, Green and Harris's amazing survey
paper \cite{EGH}. They connect the result to a host of interesting
mathematics, including the Riemann-Roch Theorem, residues and
homological algebra. Their exposition culminates in the assertion that the theorem is
equivalent to the statement that polynomial rings are Gorenstein. 

My approach to the Braikenridge-Maclaurin Theorem was inspired by
thinking about hyperplane arrangements. A good introduction to these
objects from an algebraic and topological viewpoint is the book by
Orlik and Terao \cite{TO}. For a more combinatorial viewpoint, see
Stanley's lecture notes \cite{Stanley}.

One way to view what we've done is to note that if $\Gamma$ is a
complete intersection -- a codimension $d$ variety (or, more generally,
scheme) defined by the vanishing of $d$ polynomials -- and $\Gamma$ is
made up of two subvarieties, then special properties of one subvariety
are reflected in special properties of the other subvariety. This
point of view leads to the beautiful subject of liaison theory. The last chapter of
Eisenbud \cite{Eisenbrick} introduces this advanced topic in
Commutative Algebra; more details can be
found in Migliore and Nagel's notes \cite{MN}. 


\begin{exercise} The following exercises are roughly in order of
  increasing difficulty. \label{ex:fun}

\begin{enumerate}

 \item Pascal's Theorem says that if a regular hexagon is inscribed in
  a circle then the 3 pairs of opposite edges lie on lines that
  intersect in 3 collinear points. Which line do the three points lie
  on? Is it surprising that it doesn't matter where in the plane the circle is
  centered? 

\item When working with lines in $\P^2$ it is desirable to have a
  quick way to compute their intersection points. Show that the lines
  $a_1x+b_1y+c_1z=0$ and $a_2x+b_2y+c_2z=0$ meet in the point
  $[a_3:b_3:c_3]$ where $$\langle a_3, b_3,c_3 \rangle = \langle a_1,
  b_1, c_1 \rangle \times \langle a_2, b_2, c_2 \rangle. $$ Interpret
  the result in terms of the geometry of 3-dimensional space. Also
  describe how to use this result to compute the intersection of two
  lines in $\R^2$. \label{ex:crossp}

\item There is an interesting duality between points and lines in
  $\P^2$. Fixing a nondegenerate inner product on 3-dimensional space, we
  define the dual line $\check{P}$ to a point $P \in \P^2$  to be the
  projectivization of the 2-dimensional subspace orthogonal to the
  1-dimensional subspace corresponding to $P$. Similarly, if $L$ is a
  line in $\P^2$, it corresponds to a 2-dimensional subspace in 3-dimensional space and
  we define the dual point $\check{L}$ to be the projectivization of the
  1-dimensional subspace orthogonal to this subspace. \\
(a) Show that a line $L$ in $\P^2$ goes through two points $P_1 \neq 
P_2$ if and only if the dual point $\check{L}$ lies on the intersection
of the two dual lines $\check{P_1}$ and $\check{P_2}$. \\
(b) Use part (a) and Exercise \ref{ex:fun}.\ref{ex:crossp} to develop
a cross product formula for the line through 2 points in
$\P^2$. Extend the formula to compute the equation for a line through
2 points in $\R^2$. \\
(c) It turns out that the duals of all the tangent lines to an
irreducible conic $C$ form a collection of points lying on a dual
irreducible conic $\check{C}$, and vice-versa (see Bachelor, Ksir and
Traves \cite{BKT} for details). Show that dualizing Pascal's Theorem
gives Brian\c{c}on's Theorem: If an irreducible conic is inscribed in
a hexagon, then the three lines joining pairs of opposite vertices
intersect at a single point\footnote{ Like B\'ezout, Charles Julien
Brian\c{c}on (1783-1864) was a professor at a French military academy.
The French military of the $19^\text{th}$ century seems to have played an interesting role
in supporting the development and teaching of mathematics.}. 

\item Provide a proof for one of the assertions in the paper: any set
  of $k$ collinear points fails to impose $k-(d+1)$ conditions on
  forms of degree $d \leq k-1$.

\item Establish the following result due to M\"obius \cite{Mobius}
  using the Cayley-Bacharach Theorem. Consider two polygons $P_1$ and
  $P_2$, each with $m$ edges, inscribed in a conic, and associate one
  edge from $P_1$ with one edge from $P_2$. Working counterclockwise
  in each polygon, associate the other edges of $P_1$ with the edges
  of $P_2$. Extending these edges to lines, M\"obius proved that if
  $m-1$ of the intersections of pairs of corresponding edges lie on a
  line then the last pair of corresponding edges also meets in a point
  on this line.

\item Establish the following result due to Katz \cite[Theorem
  3.3]{Katz}, his Mystic $2d$-Gram Theorem. If $d$ red lines and $d$
  blue lines intersect in $d^2$ points and if $2d$ of these points lie
  on an irreducible conic then there is a unique curve of degree $k-2$ through the
  other $d^2-2d$ intersection points. Katz's interesting paper \cite{Katz}
  contains several open problems.  \label{ex:katz} 

\item Use the Cayley-Bacharach Theorem to show that if two degree-5
  curves meet in 25 points, 10 of which lie on an irreducible degree-3
  curve, then there is a unique degree-4 curve through the other 15
  points. Also convince yourself that the hypotheses of this exercise can
  actually occur. 

\item If a degree-8 curve meets a degree-9 curve in 72 points and if
  17 of these points lie on an irreducible degree-3 curve, then what
  is the dimension of the family of degree-9 curves through the
  remaining 55 points? Convince yourself that the hypotheses of this
  exercise can actually occur.

\item Use the $8 \Rightarrow 9$ Theorem to show that the group law on
  an elliptic curve is associative. 

\item In general you might expect that
  if $X \subset \P^n$ then $\text{dim } \Sec(X) = \text{
    min}(2\text{dim}(X), n)$. Varieties $X$ where this inequality
  fails to hold are called {\em defective}. Check that $\Sec(X_{1^6})$ is
  not defective: it has dimension 25.

\end{enumerate}
\end{exercise} 

\end{section}

\bibliographystyle{plain}
\bibliography{workingpaper.bib}

\end{document}